\documentclass[11pt]{article}
\usepackage{epic,latexsym,amssymb}
\usepackage{amssymb,amscd,amsmath,amsfonts,amsthm}
\usepackage{color}
\usepackage{tikz}
\usepackage{comment}
\usepackage{lineno}
\usepackage{pstricks,pst-node,pst-plot}
\usepackage[ruled,vlined,linesnumbered]{algorithm2e}
\usetikzlibrary{patterns}
\usepackage{geometry,graphicx}

\usepackage{tikz,multicol,bm}

\usepackage{enumerate}
\usepackage[ruled,vlined,linesnumbered]{algorithm2e}

\usepackage[english]{babel}
\usepackage{diagbox,comment,xparse}

\usetikzlibrary{calc}
\def\cp{\,\square\,}

\usepackage{amsmath}

\newtheorem{theorem}{Theorem}
\newtheorem{lemma}[theorem]{Lemma}

\newtheorem{proposition}[theorem]{Proposition}

\newtheorem{remark}{Remark}

\newcommand{\2}{ \vspace{0.2cm} }
\newcommand{\1}{ \vspace{0.1cm} }

\let\oldenumerate\enumerate
\renewcommand{\enumerate}{
  \oldenumerate
  \setlength{\itemsep}{0pt}
  \setlength{\parskip}{0pt}
  \setlength{\parsep}{0pt}
}

\begin{document}

%%\linenumbers

\title{On polluted bootstrap percolation in Cartesian grids}

\author{Bo\v{s}tjan Bre\v{s}ar$^{a,b}$, Jaka Hed\v zet$^{a,b}$, \, and \,  Michael A. Henning$^{c}$ \\ \\
$^a$ Faculty of Natural Sciences and Mathematics \\
University of Maribor, Slovenia \\
$^b$ Institute of Mathematics, Physics and Mechanics, \\
Ljubljana, Slovenia\\ 
\small \tt Email: bostjan.bresar@um.si \\
\small \tt Email: jaka.hedzet@imfm.si \\
\\
$^{c}$ Department of Mathematics and Applied Mathematics \\
University of Johannesburg, South Africa\\
\small \tt Email: mahenning@uj.ac.za
}

\date{}
\maketitle

 \begin{abstract}
 \noindent
 Given a graph $G$ and assuming that some vertices of $G$ are infected, the $r$-neighbor bootstrap percolation rule makes an uninfected vertex $v$ infected if $v$ has at least $r$ infected neighbors. The $r$-percolation number, $m(G, r)$, of $G$ is the minimum cardinality of a set of initially infected vertices in $G$ such that after continuously performing the $r$-neighbor bootstrap percolation rule each vertex of $G$ eventually becomes infected. In this paper, we continue the study of polluted bootstrap percolation introduced and studied by Gravner and McDonald [Bootstrap percolation in a polluted environment. J.\ Stat\ Physics 87 (1997) 915--927] where in this variant some vertices are permanently in the non-infected state. We study an extremal (combinatorial) version of the bootstrap percolation problem in a polluted environment, where our main focus is on the class of grid graphs, that is, the Cartesian product $P_m \cp P_n$ of two paths $P_m$ and $P_n$ on $m$ and $n$ vertices, respectively. Given a number of polluted vertices in a Cartesian grid we establish a closed formula for the minimum $2$-neighbor bootstrap percolation number of the polluted grid, and obtain a lower bound for the other extreme.  
 \\[2mm]
 \end{abstract}

\baselineskip=0.20in

\section{Introduction}

Bootstrap percolation originated in the work of Chalupa, Leath and Reich in the context of physics of ferromagentism~\cite{cha-1979}. The process involving bootstrap percolation seems to be ubiquitous, since it is also closely related to cellular automata~\cite{mor-2017}, weak saturation~\cite{bol-1968} as well as to the process of obtaining the $P_3$-convex hull~\cite{cen-2010} to mention a few well known contexts in which it appears.
Given an integer $r \ge 2$, the \emph{$r$-neighbor bootstrap percolation} process is an update rule in a given graph $G$, where vertices have two possible states: \emph{infected} (black) or \emph{uninfected} (white). From an initial set of infected vertices, an uninfected vertex with at least $r$ infected neighbors becomes infected, while infected vertices never change their state;  the central question is whether all vertices of $G$ become infected eventually when this rule is continuously applied.

Bootstrap percolation was studied in arbitrary graphs and in various graph classes, while arguably the most popular environment for studying this process are Cartesian grid networks and similar product-like structures; see~\cite{BP-1998} for an early study of spreading random disease in square grids. In majority of these studies a parameter $p$ describing the probability that vertices are initially infected is involved, but one can consider also the extremal version of the problem; for instance, see~\cite{Bol-2006}. In this problem, one looks for a smallest {\em $r$-percolating set} in a graph $G$, which is a set of vertices in $G$ whose initial infection results in all vertices of $G$ becoming infected after the process is finished. For a given graph $G$ and integer $r \ge 2$, the \emph{$r$-percolation number} of $G$, denoted $m(G,r)$, is the minimum cardinality of an $r$-percolating set in $G$, that is,
\[
m(G,r)=\min \left\lbrace |S| \colon \, S \subseteq V(G), \: \mbox{$S$ is an $r$-percolating set in $G$} \right\rbrace.
\]
(A \emph{minimum $r$-percolating set} in $G$ is a $r$-percolating set $S$ of $G$ satisfying $m(G,r) = |S|$.) For the Cartesian (square) grid $G$, the problem of determining $m(G,2)$ is part of folklore; see Bollob\'{a}s~\cite{BP-1998} from 1998.  Anyway, several questions concerning $m(G,r)$ remain open, in particular for $r>2$, and for grids of bigger dimensions than $2$; see recent studies~\cite{dnr-2023,hh-2025}; see also a recent study for bounding the threshold for critical probability for the case $r=2$ in the square grid~\cite{hm-2019}.

Several variations of bootstrap percolation were considered in the literature. In this note, we are interested in the polluted bootstrap percolation, as introduced by Gravner and McDonald~\cite{GM} and studied further in~\cite{GH-2009,ghls-2025}. The additional assumption in this variant is that some vertices of a graph are always stuck in the non-infected state. Thus, to every vertex of a graph two parameters, $p$ and $q$, are assigned, describing  probabilities of infection and pollution of its vertices, respectively. The authors in~\cite{GM}
investigate constants $C$, which under the condition $q<Cp^2$ when $p\rightarrow 0$ result in the probability that all vertices of a graph are eventually (not) infected. 

In this note, we propose an extremal (combinatorial) version of the bootstrap percolation problem in polluted environment.
Let the vertices of a graph $G$, as observed in polluted environment, assume one of the three states: infected, uninfected or \emph{polluted}. If a vertex $v \in V(G)$ is polluted, then it is always in the uninfected state, even if it has enough infected neighbors to satisfy the $r$-percolation rule.
From the definitions it readily follows that $m(G,r)=m(G-v,r)$ for any polluted vertex $v$. More generally, if $A\subset V(G)$ is the set of polluted vertices in $G$, then $m(G,r)=m(G-A,r)$. In other words, considering polluted bootstrap percolation in $G$ can be translated to bootstrap percolation in the induced subgraph of $G$, where polluted vertices have been removed from $G$.

When considering the $r$-bootstrap percolation number of a graph $G$, some vertices that must be included in the initial infection do not contribute much to the spread of the infection (for instance, such vertices are leaves in a tree). The removal of such vertices could therefore reduce the minimum cardinality of a percolating set. On the other hand, vertices of large degree contribute abundantly to the spread of infection; therefore, their removal likely increases the number of vertices needed to infect $G$. Thus, a natural question arises. Given $k\in \mathbb{N}$  and assuming that $G$ has $k$ polluted vertices, what are the largest and the smallest possible value of $m(G,r)$? More precisely, for fixed integers $k\in \mathbb{N}$ and $r \ge 2$ we are looking for the values of
$$m_k^{\min}(G,r)=\min\{m(G-A,r): \, A \subseteq V(G)\, \wedge \, |A|=k \}$$ and $$m_k^{\max}(G,r)=\max\{m(G-A,r): \, A \subseteq V(G)\, \wedge \, |A|=k \}.$$
Our main focus is on the above two invariants when $G$ is the square grid, that is, the Cartesian product of two paths, and when $r=2$. The following is our main result:

\begin{theorem}
\label{thm:main-grid}
    If $G=P_m \square P_n$ is the grid of size $m \times n$ where $2\le n \le m$, then
    \[
        m_k^{\min}(G,2) =
        \begin{cases}
          \left\lceil\frac{n+m-\left\lfloor\frac{k}{n}\right\rfloor}{2}\right\rceil, & \text{if } 1\le k \le (m-n)n; \2 \\
          \left\lceil\frac{\left\lceil 2\sqrt{mn-k}\right\rceil}{2}\right\rceil, & \text{if } (m-n)n\le k \le mn.
        \end{cases}
    \]
\end{theorem}

\section{Preliminary results}

In this section, we present some preliminary results on $r$-neighbor bootstrap percolation in graphs with minimum degree at least~$r$.

\begin{lemma}\label{lem:G-v}
    If $G$ is a graph with $\delta(G) \ge r$, then $m(G-x,r) \ge m(G,r)$ for any $x \in V(G)$.
\end{lemma}
\begin{proof}
    Let $\delta(G) \ge r$ and let $x \in V(G)$. Suppose that $m(G-x,r) < m(G,r)$ and let $S$ be a minimum $r$-percolating set of $G-x$. Now, consider the graph $G$ and let the vertices of $S$ be infected. Then eventually $S$ infects the whole $G-x$, after which $x$ also gets infected, since $\deg_G(x) \ge r$. Thus, $S$ is an $r$-percolating set of $G$ of size $|S| < m(G,r)$, a contradiction.
\end{proof}

To see that the condition $\delta(G) \ge r$ in Lemma~\ref{lem:G-v} is necessary,
let $r=2$ and let $G$ be the complete bipartite graph $K_{1,k}$ for some $k \ge 3$. Then $m(G,r)=k$, since all leaves must be initially infected. By removing any leaf $v$, we get $m(G-v,r)=k-1$, as the remaining leaves still form a minimum $r$-percolating set.

From Lemma~\ref{lem:G-v} we infer the following result.

\begin{theorem}\label{thm:general}
If $G$ is a graph with $\delta(G) \ge r$, then $m(G-A,r)\ge m(G,r)$ for any independent set $A \subset V(G)$.
\end{theorem}
\begin{proof}
    Let $A = \{v_1,v_2, \ldots ,v_k\}$ be an independent set of $G$. By repeatedly applying Lemma~\ref{lem:G-v} we get $$m(G,r) \le m(G-v_1,r) \le m(G-\{v_1,v_2\},r) \le \cdots \le m(G-A,r).$$
\end{proof}

\section{Cartesian grids}

Let us recall some definitions.
The \emph{Cartesian product} $G \cp H$ of two graphs $G$ and $H$ is the graph whose vertex set is $V(G) \times V(H)$, and where two vertices $(g_1,h_1)$ and $(g_2,h_2)$ are adjacent in $G \cp H$ if either $g_1=g_2$ and $h_1h_2 \in E(H)$, or $h_1=h_2$ and $g_1g_2 \in E(G)$. In the case of the Cartesian product of two paths $P_m$ and $P_n$ for some integers $m,n \ge 2$, the resulting graph $P_m \square P_n$ is a (\emph{square}) \emph{grid}. A~\emph{torus} is the Cartesian product $C_m \square C_n$ of any two cycles.

For obtaining the lower bound on $m_k^{\min}(G,2)$ for the graphs $G=P_m \cp P_n$ we will use the perimeter approach initiated by Bollob\'{a}s in \cite{Bol-2006}. In this approach, the vertices of $G$ are replaced by unit squares, and two squares share an edge whenever the corresponding vertices are adjacent. Whenever a square gets infected by at least two other squares, the perimeter of the infected domain does not increase, therefore, the perimeter of the initial infection must be at least as large as the perimeter of the total grid, which is $2m+2n$. Note that the perimeter argument can be applied on any subgraph of the grid. Our goal is to remove $k$ squares from the grid in such a way that the total perimeter of the remaining graph $G'$ is the smallest possible. This is closely related to constructing a shape with the smallest perimeter from $t=nm-k$ squares. For this purpose, we need the following lemma.

\begin{lemma}\label{lem:4n-2e}
    If there are $n$ unit squares in the plane, which can pairwise intersect in the way that two squares  share an edge, then the total perimeter of the acquired shape is $4n -2e$, where $e$ is the number of shared edges.
\end{lemma}

Since we want to minimize the total perimeter, the value of $e$ in the statement of Lemma~\ref{lem:4n-2e} should be maximized.

\begin{proposition}\label{prop:x^2+r}
If $p(n)$ is the minimum perimeter of a shape in the plane obtained from $n$ unit squares, which can pairwise intersect in the way that two squares share an edge, then for every $x^2\le n < (x+1)^2$ it follows that $4x\le p(n) \le 4(x+1)$. Moreover,
        $$p(x^2+r)=
    \begin{cases}
        4x, & r=0; \\
        4x+2, & 0 < r \le x; \\
        4x+4, & x < r \le 2x.
    \end{cases}
    $$
\end{proposition}
\begin{proof}
Let $H$ be a shape in the plane obtained from $n$ unit squares, which can pairwise intersect in the way that two squares share an edge, such that the perimeter of $H$ is $p(n)$. By applying Lemma~\ref{lem:4n-2e}, we infer that $H$ has to be close to a rectangular shape, while the fact that the product $u \times v$ is the smallest among all pairs of integers $u$ and $v$ with the fixed sum $u+v$ when $|u-v|$ is minimized, the shape of $H$ needs to be as close to a square as possible. In particular, if $n=x^2$, then $p(n)=4x$. When the shape cannot be a square, it is immediate (again using Lemma~\ref{lem:4n-2e}) that the shape must be connected and have no holes. Now let $x$ be the largest integer such that $n = x^2+r$, where $r$ is a positive integer. Pick the smallest rectangle containing $H$, and let $a \times b$ be its dimension.
The perimeter of $H$ is at least $2a + 2b$, since when one follows the edge of $H$, one needs to go all the way from left to right and from top to bottom twice.
    On the other hand, $ab \ge n$ because there are at least $n$ tiles inside the rectangle. Hence,
    \[
    p(n) \ge 2a + 2b \ge 2a + \frac{2n}{a}.
    \]
    The right side of the above inequality has a unique minimum at $a = \sqrt{n}$, but $a$ is an integer, so the right side is minimized when $a \in \{x,x+1\}$. First let $r \le x$. Then we have
    \[
    p(n) \ge 2x + \frac{2n}{x} = \frac{4x^2 + 2r}{x}
    \]
    as it is easy to see that $a = x + 1$ gives a larger value. If $r = 0$, this gives $p(n) \ge 4x$.
    If $0<r \le x$, we have $p(n) > 4x$, but as $p(n)$ is even, we have $p(n) \ge 4x + 2$.
    Finally if $r > x$, we should take $a=x+1$, as it gives a smaller lower bound on $p(n)$ than in the case when $a=x$. Since $r>x$, we get $n>x^2+x$. Therefore from $ab > x^2+x$ it follows that $b > x$ (or rather $b \ge x+1)$. This means that $p(n) \ge 2a + 2b \ge 2(x+1)+2(x+1) = 4x+4$. Since there exists shapes obtaining these bounds for all three possibilities (see~Figure~\ref{fig:examples of smallest perimeter}), the proof is complete.
\end{proof}

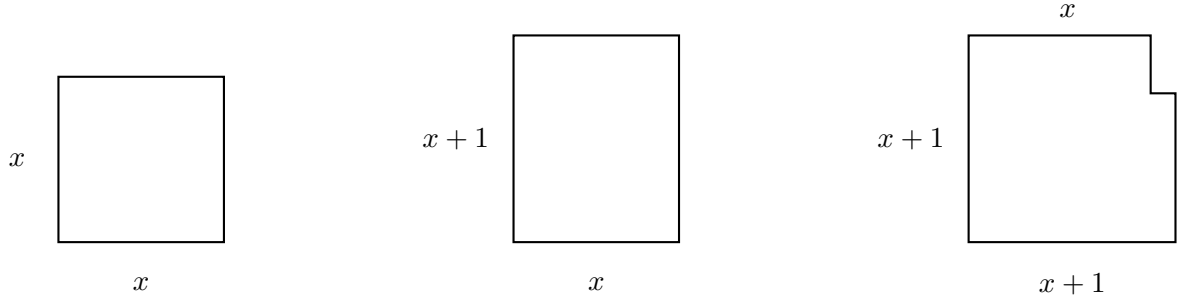
\begin{figure}
    \centering
    \begin{tikzpicture}[scale=1.1]

  % First shape: square x x (x = 2)
  \draw[thick]
    (0,0) -- (0,2) -- (2,2) -- (2,0) -- cycle;
  \node at (1,-0.5) {$x$};
  \node at (-0.5,1) {$x$};

  % Second shape: rectangle x x+1
  \begin{scope}[xshift=5.5cm]
    \draw[thick]
      (0,0) -- (0,2.5) -- (2,2.5) -- (2,0) -- cycle;
    \node at (1,-0.5) {$x$};
    \node at (-0.7,1.25) {$x+1$};
  \end{scope}

  % Third shape: (x+1) x (x+1) with one unit square removed
  \begin{scope}[xshift=11cm]
    \draw[thick]
      (0,0) -- (0,2.5) -- (2.2,2.5) -- (2.2,1.8) -- (2.5,1.8) -- (2.5,0) -- cycle;
    \node at (1.25,-0.5) {$x+1$};
    \node at (-0.7,1.25) {$x+1$};
     \node at (1.2,2.8) {$x$};
  \end{scope}

\end{tikzpicture}

    \caption{Examples with perimeters $4x$, $4x+2$ and $4x+4$, respectively}
    \label{fig:examples of smallest perimeter}
\end{figure}

\begin{proposition}\label{prp:xy+r}
   If $p(n)$ is the minimum perimeter of a shape in the plane obtained from $n$ unit squares, which can pairwise intersect in the way that two squares share an edge, such that the height of the shape is at most $x$, and $x^2 \le xy \le n < x(y+1)$, then  $$p(xy+r)=
    \begin{cases}
        2x+2y, & r=0;\\
        2x+2y+2, & 0 < r \le x.
    \end{cases}
    $$
\end{proposition}

\begin{proof}
    By using a similar reasoning as in the earlier proofs, we infer that a shape with the smallest perimeter is connected without holes and is contained within the smallest possible rectangle $a \times b$. Since the height is at most $x$, it follows that the shape is contained within a rectangle $x \times b$, therefore it has perimeter of at least $2x + 2b$. If $n = xy$, then the perimeter is exactly $2x+2y$. If $xy < n \le x(y+1)$, then the perimeter is $2x+2y +2$.
\end{proof}

The shapes with minimum perimeter provide a lower bound for the $2$-bootstrap percolation of the corresponding graph. The next result is obtained by combining the previous two results and having in mind that every unit square adds at most $4$ to the total perimeter.

\begin{lemma}
\label{lem:rectangle}
    Let $G=P_m \square P_n$ be the grid of size $m \times n$, where $2\le n \le m$, and let $t=mn-k$, where $k > 0$. If $G'$ is a subgraph of $G$ obtained from $G$ by removing $k$ of its vertices, then
    \[
    m(G',2) \geq
    \begin{cases}
        %\left\lceil\frac{n+m-\ell}{2}\right\rceil, & t=n(m-\ell)\textrm{ where }m-\ell>n\\
        \left\lceil\frac{n+m-\ell}{2}\right\rceil, & t=n(m-\ell-1)+r, \, 0< r\le n \, \textrm{ and where }m-\ell>n;\\
        x, & t=x^2, x \le n;\\
        x+1, & t=x^2+r, x < n, r \le 2x;
    \end{cases}
    \]
    where $\ell = \lfloor \frac{k}{n}\rfloor.$
\end{lemma}

\begin{proof}
    Note that for $t\le n^2$, the result follows from Proposition~\ref{prop:x^2+r}, which resolves the second and the third line in the right side of the stated inequality.

    If $t>n^2$, the results follows by using Proposition~\ref{prp:xy+r}. Indeed, let $t=n(m-\ell-1) + r$, where $\ell = \lfloor \frac{n}{k}\rfloor$ and $r\le n$. If $r=n$, then $t=n(m-\ell)$ and we let $x=n$ and $y=(m-\ell)$. Thus, the smallest perimeter is $p(xy)=2n+2(m-\ell)$. Therefore in this case, $m(G',2) \ge \left\lceil\frac{n+(m-\ell)}{2}\right\rceil$. For $0<r<n$, we let $x=n$ and $y=m-\ell-1$, and so $t=xy+r$. The smallest perimeter is therefore in this case $2n+2(m-\ell-1) +2 = 2n+2m-2\ell$ and $m(G',2) \ge \left\lceil\frac{n+(m-\ell)}{2}\right\rceil.$
\end{proof}

Next, we focus on obtaining the upper bound on $m_k^{\min}(P_m\cp P_n,2)$. We again distinguish two possibilities with respect to $k$, which this time are presented in two separate auxiliary results.

\begin{lemma}\label{lem:rectangle grids}
    Let $G=P_m \square P_n$ be the grid of size $m \times n$ where $2\le n \le m$. If $1 \le k \le (m-n)n$, then
    \[
        m_k^{\min}(G,2) \le \left\lceil\frac{n+m-\ell}{2}\right\rceil,    \]
        where $\ell = \lfloor\frac{k}{n}\rfloor$.
\end{lemma}
\begin{proof}
First we will provide a construction of a polluted set $P$, and then a construction of a percolating set $S$ of $G'=G-P$. Denote $V(P_m)=[m]$, $V(P_n)=[n]$ and let $\ell = \lfloor\frac{k}{n}\rfloor$. If $k = \ell n$, then let
\[
P = \lbrace m-\ell +1, m-\ell+2,\ldots, m \rbrace \times [n],
\]
while if $k > \ell n$, then let
\[
\begin{array}{lcl}
P & = & \lbrace m-\ell +1, m-\ell+2,\ldots, m \rbrace \times [n] \,\, \cup  \1 \\
& & \, \, \{(m-\ell,n),(m-\ell,n-1),\ldots, (m-\ell,n-(k-\ell n)+1) \}.
\end{array}
\]
In other words, $P$ consists of the final $\ell$ columns of $G$ and, if $k > \ell n$, then the remaining $k-\ell n$ vertices from the top of the $(m-\ell)$-th column. Notice that if $k = \ell n$, then the remaining graph $G'=G-P$ has $m-\ell$ full columns (that contain no polluted vertex), and if $k > \ell n$, then the graph $G'$ has $m-\ell -1$ full columns and, additionally, vertices $(m-\ell,i)$, where $1\le i \le n-(k-\ell n)$. When $m = 8$ and $n = 5$, the polluted $m \times n$ grid with $k=8$ and $k=11$ is illustrated in Figure~\ref{fig:grid small k}(a) and~\ref{fig:grid small k}(b), respectively, where dotted vertices (associated with the set $P$ defined earlier) are polluted and shaded vertices present initial infection.

Consider the path $C$ in $G$ with $(1,n)$ as one of its endvertices, containing all vertices of the first column and the first row in $G$.  Now, color the vertices of $C$ by alternating black and white vertices, beginning with $(1,n)$ as black. Let $S$ consist of all black vertices of $C$ together with $(m-\ell,1)$ in the case when $m-\ell +n$ is odd. Since $C$ contains $m-\ell +n$ vertices, it follows that $|S|=\left\lceil\frac{m+n-\ell}{2}\right\rceil$. The set $S$ is a percolating set of $G'$. Indeed, in the first round the whole set $C$ becomes infected. After that, we can see that the $2$nd column becomes infected in order $(2,2),(2,3), \ldots,(2,n)$, then the third and so on. This gives us an upper bound on $m_k^{\min}(G,2)$.
\end{proof}

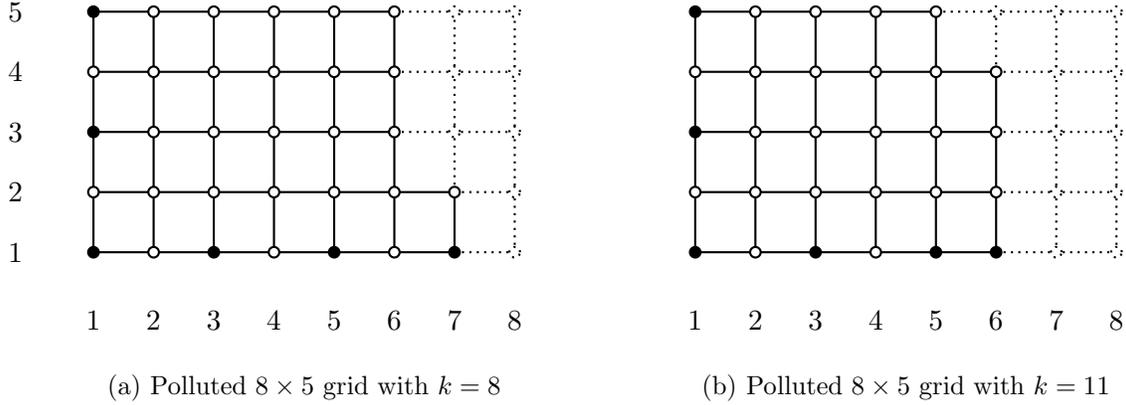
\begin{figure}[ht!]
\begin{center}
\begin{tikzpicture}[scale=0.8,style=thick,x=1cm,y=1cm]
\def\vr{2.5pt} % \vr = vertex radius;
% define vertices

\path (-1,1) coordinate (a);
\path (-1,2) coordinate (b);
\path (-1,3) coordinate (c);
\path (-1,4) coordinate (d);
\path (-1,5) coordinate (e);
\path (0,0.2) coordinate (1);
\path (1,0.2) coordinate (2);
\path (2,0.2) coordinate (3);
\path (3,0.2) coordinate (4);
\path (4,0.2) coordinate (5);
\path (5,0.2) coordinate (6);
\path (6,0.2) coordinate (7);
\path (7,0.2) coordinate (8);

\path (0,1) coordinate (a1);
\path (0,2) coordinate (b1);
\path (0,3) coordinate (c1);
\path (0,4) coordinate (d1);
\path (0,5) coordinate (e1);

\path (1,1) coordinate (a2);
\path (1,2) coordinate (b2);
\path (1,3) coordinate (c2);
\path (1,4) coordinate (d2);
\path (1,5) coordinate (e2);

\path (2,1) coordinate (a3);
\path (2,2) coordinate (b3);
\path (2,3) coordinate (c3);
\path (2,4) coordinate (d3);
\path (2,5) coordinate (e3);

\path (3,1) coordinate (a4);
\path (3,2) coordinate (b4);
\path (3,3) coordinate (c4);
\path (3,4) coordinate (d4);
\path (3,5) coordinate (e4);

\path (4,1) coordinate (a5);
\path (4,2) coordinate (b5);
\path (4,3) coordinate (c5);
\path (4,4) coordinate (d5);
\path (4,5) coordinate (e5);

\path (5,1) coordinate (a6);
\path (5,2) coordinate (b6);
\path (5,3) coordinate (c6);
\path (5,4) coordinate (d6);
\path (5,5) coordinate (e6);

\path (6,1) coordinate (a7);
\path (6,2) coordinate (b7);
\path (6,3) coordinate (c7);
\path (6,4) coordinate (d7);
\path (6,5) coordinate (e7);

\path (7,1) coordinate (a8);
\path (7,2) coordinate (b8);
\path (7,3) coordinate (c8);
\path (7,4) coordinate (d8);
\path (7,5) coordinate (e8);

%  edges
\draw (a1)--(b1)--(c1)--(d1)--(e1);
\draw (a2)--(b2)--(c2)--(d2)--(e2);
\draw (a3)--(b3)--(c3)--(d3)--(e3);
\draw (a4)--(b4)--(c4)--(d4)--(e4);
\draw (a5)--(b5)--(c5)--(d5)--(e5);
\draw (a6)--(b6)--(c6)--(d6)--(e6);
\draw (a7)--(b7);
\draw[dotted] (b7)--(c7)--(d7)--(e7);
\draw[dotted] (a8)--(b8)--(c8)--(d8)--(e8);

\draw (a1)--(a2)--(a3)--(a4)--(a5)--(a6)--(a7);
\draw[dotted] (a7)--(a8);
\draw (b1)--(b2)--(b3)--(b4)--(b5)--(b6)--(b7);
\draw[dotted] (b7)--(b8);
\draw (c1)--(c2)--(c3)--(c4)--(c5)--(c6);
\draw[dotted] (c6)--(c7)--(c8);
\draw (d1)--(d2)--(d3)--(d4)--(d5)--(d6);
\draw[dotted] (d6)--(d7)--(d8);
\draw (e1)--(e2)--(e3)--(e4)--(e5)--(e6);
\draw[dotted] (e6)--(e7)--(e8);

% vertices
\draw (a1) [fill=black] circle (\vr);
\draw (b1) [fill=white] circle (\vr);
\draw (c1) [fill=black] circle (\vr);
\draw (d1) [fill=white] circle (\vr);
\draw (e1) [fill=black] circle (\vr);

\draw (a2) [fill=white] circle (\vr);
\draw (b2) [fill=white] circle (\vr);
\draw (c2) [fill=white] circle (\vr);
\draw (d2) [fill=white] circle (\vr);
\draw (e2) [fill=white] circle (\vr);

\draw (a3) [fill=black] circle (\vr);
\draw (b3) [fill=white] circle (\vr);
\draw (c3) [fill=white] circle (\vr);
\draw (d3) [fill=white] circle (\vr);
\draw (e3) [fill=white] circle (\vr);

\draw (a4) [fill=white] circle (\vr);
\draw (b4) [fill=white] circle (\vr);
\draw (c4) [fill=white] circle (\vr);
\draw (d4) [fill=white] circle (\vr);
\draw (e4) [fill=white] circle (\vr);

\draw (a5) [fill=black] circle (\vr);
\draw (b5) [fill=white] circle (\vr);
\draw (c5) [fill=white] circle (\vr);
\draw (d5) [fill=white] circle (\vr);
\draw (e5) [fill=white] circle (\vr);

\draw (a6) [fill=white] circle (\vr);
\draw (b6) [fill=white] circle (\vr);
\draw (c6) [fill=white] circle (\vr);
\draw (d6) [fill=white] circle (\vr);
\draw (e6) [fill=white] circle (\vr);

\draw (a7) [fill=black] circle (\vr);
\draw (b7) [fill=white] circle (\vr);
\draw[dotted] (c7) [fill=white] circle (\vr);
\draw[dotted] (d7) [fill=white] circle (\vr);
\draw[dotted] (e7) [fill=white] circle (\vr);

\draw[dotted] (a8) [fill=white] circle (\vr);
\draw[dotted] (b8) [fill=white] circle (\vr);
\draw[dotted] (c8) [fill=white] circle (\vr);
\draw[dotted] (d8) [fill=white] circle (\vr);
\draw[dotted] (e8) [fill=white] circle (\vr);

% text
\draw[anchor = east] (a) node {1};
\draw[anchor = east] (b) node {2};
\draw[anchor = east] (c) node {3};
\draw[anchor = east] (d) node {4};
\draw[anchor = east] (e) node {5};
\draw[anchor = north] (1) node {1};
\draw[anchor = north] (2) node {2};
\draw[anchor = north] (3) node {3};
\draw[anchor = north] (4) node {4};
\draw[anchor = north] (5) node {5};
\draw[anchor = north] (6) node {6};
\draw[anchor = north] (7) node {7};
\draw[anchor = north] (8) node {8};

%%%%%%%%%
\draw (3.5,-1.25) node {{\small (a) Polluted $8\times 5$ grid with $k=8$}};
%%%%%%%%

\begin{scope}[shift={(10,0)}]

\path (0,0.2) coordinate (1);
\path (1,0.2) coordinate (2);
\path (2,0.2) coordinate (3);
\path (3,0.2) coordinate (4);
\path (4,0.2) coordinate (5);
\path (5,0.2) coordinate (6);
\path (6,0.2) coordinate (7);
\path (7,0.2) coordinate (8);

\path (0,1) coordinate (a1);
\path (0,2) coordinate (b1);
\path (0,3) coordinate (c1);
\path (0,4) coordinate (d1);
\path (0,5) coordinate (e1);

\path (1,1) coordinate (a2);
\path (1,2) coordinate (b2);
\path (1,3) coordinate (c2);
\path (1,4) coordinate (d2);
\path (1,5) coordinate (e2);

\path (2,1) coordinate (a3);
\path (2,2) coordinate (b3);
\path (2,3) coordinate (c3);
\path (2,4) coordinate (d3);
\path (2,5) coordinate (e3);

\path (3,1) coordinate (a4);
\path (3,2) coordinate (b4);
\path (3,3) coordinate (c4);
\path (3,4) coordinate (d4);
\path (3,5) coordinate (e4);

\path (4,1) coordinate (a5);
\path (4,2) coordinate (b5);
\path (4,3) coordinate (c5);
\path (4,4) coordinate (d5);
\path (4,5) coordinate (e5);

\path (5,1) coordinate (a6);
\path (5,2) coordinate (b6);
\path (5,3) coordinate (c6);
\path (5,4) coordinate (d6);
\path (5,5) coordinate (e6);

\path (6,1) coordinate (a7);
\path (6,2) coordinate (b7);
\path (6,3) coordinate (c7);
\path (6,4) coordinate (d7);
\path (6,5) coordinate (e7);

\path (7,1) coordinate (a8);
\path (7,2) coordinate (b8);
\path (7,3) coordinate (c8);
\path (7,4) coordinate (d8);
\path (7,5) coordinate (e8);

%  edges
\draw (a1)--(b1)--(c1)--(d1)--(e1);
\draw (a2)--(b2)--(c2)--(d2)--(e2);
\draw (a3)--(b3)--(c3)--(d3)--(e3);
\draw (a4)--(b4)--(c4)--(d4)--(e4);
\draw (a5)--(b5)--(c5)--(d5)--(e5);
\draw (a6)--(b6)--(c6)--(d6);
\draw[dotted] (d6)--(e6);
\draw[dotted] (a7)--(b7)--(c7)--(d7)--(e7);
\draw[dotted] (a8)--(b8)--(c8)--(d8)--(e8);

\draw (a1)--(a2)--(a3)--(a4)--(a5)--(a6);
\draw[dotted] (a6)--(a7)--(a8);
\draw (b1)--(b2)--(b3)--(b4)--(b5)--(b6);
\draw[dotted] (b6)--(b7)--(b8);
\draw (c1)--(c2)--(c3)--(c4)--(c5)--(c6);
\draw[dotted] (c6)--(c7)--(c8);
\draw (d1)--(d2)--(d3)--(d4)--(d5)--(d6);
\draw[dotted] (d6)--(d7)--(d8);
\draw (e1)--(e2)--(e3)--(e4)--(e5);
\draw[dotted] (e5)--(e6)--(e7)--(e8);

% vertices
\draw (a1) [fill=black] circle (\vr);
\draw (b1) [fill=white] circle (\vr);
\draw (c1) [fill=black] circle (\vr);
\draw (d1) [fill=white] circle (\vr);
\draw (e1) [fill=black] circle (\vr);

\draw (a2) [fill=white] circle (\vr);
\draw (b2) [fill=white] circle (\vr);
\draw (c2) [fill=white] circle (\vr);
\draw (d2) [fill=white] circle (\vr);
\draw (e2) [fill=white] circle (\vr);

\draw (a3) [fill=black] circle (\vr);
\draw (b3) [fill=white] circle (\vr);
\draw (c3) [fill=white] circle (\vr);
\draw (d3) [fill=white] circle (\vr);
\draw (e3) [fill=white] circle (\vr);

\draw (a4) [fill=white] circle (\vr);
\draw (b4) [fill=white] circle (\vr);
\draw (c4) [fill=white] circle (\vr);
\draw (d4) [fill=white] circle (\vr);
\draw (e4) [fill=white] circle (\vr);

\draw (a5) [fill=black] circle (\vr);
\draw (b5) [fill=white] circle (\vr);
\draw (c5) [fill=white] circle (\vr);
\draw (d5) [fill=white] circle (\vr);
\draw (e5) [fill=white] circle (\vr);

\draw (a6) [fill=black] circle (\vr);
\draw (b6) [fill=white] circle (\vr);
\draw (c6) [fill=white] circle (\vr);
\draw (d6) [fill=white] circle (\vr);
\draw[dotted] (e6) [fill=white] circle (\vr);

\draw[dotted] (a7) [fill=white] circle (\vr);
\draw[dotted] (b7) [fill=white] circle (\vr);
\draw[dotted] (c7) [fill=white] circle (\vr);
\draw[dotted] (d7) [fill=white] circle (\vr);
\draw[dotted] (e7) [fill=white] circle (\vr);

\draw[dotted] (a8) [fill=white] circle (\vr);
\draw[dotted] (b8) [fill=white] circle (\vr);
\draw[dotted] (c8) [fill=white] circle (\vr);
\draw[dotted] (d8) [fill=white] circle (\vr);
\draw[dotted] (e8) [fill=white] circle (\vr);

% text

\draw[anchor = north] (1) node {1};
\draw[anchor = north] (2) node {2};
\draw[anchor = north] (3) node {3};
\draw[anchor = north] (4) node {4};
\draw[anchor = north] (5) node {5};
\draw[anchor = north] (6) node {6};
\draw[anchor = north] (7) node {7};
\draw[anchor = north] (8) node {8};

%%%%%%%%%
\draw (3.5,-1.25) node {{\small (b) Polluted $8\times 5$ grid with $k=11$}};
%%%%%%%%

\end{scope}

\end{tikzpicture}
\end{center}

\caption{Polluted $8\times 5$ grid with $k=8$ and $k=11$, respectively, in the proof of Lemma~\ref{lem:rectangle grids}}
\label{fig:grid small k}
\end{figure}

\begin{lemma}\label{lem:square grids}
    If $G=P_m \square P_n$ is the Cartesian grid of size $m \times n$ where $2\le n \le m$, and $(m-n)n \le k \le mn$, then
    \[
        m_k^{\min}(G,2) \le \left\lceil\frac{1}{2}\left\lceil 2\sqrt{t} \, \right\rceil\right\rceil,
    \]
where $t =mn-k$.
\end{lemma}

\begin{proof}
Note that the construction from the proof of Lemma~\ref{lem:rectangle grids} for the case $k=(m-n)n$, where $P$ be the set of the final $\ell$ columns, gives the square grid of size $n \times n$, which we denote by $G_s$ (clearly, $G_s=G-P$).
In this case, $t=n^2$, and we get $m_k^{\min}(G,2) \le \left\lceil\frac{1}{2}\left\lceil 2\sqrt{t}\right\rceil\right\rceil=n$.  Now, let $k>(m-n)n$.
    Furthermore, let $k_1 = k -(m-n)n \ge 0$ and denote as $P_s$ the set of polluted vertices of size $k_1$ that are not in $P$. From this point onward, when $k_1 > 0$, we want to keep the shape of the remaining graph $G'=G_s-P_s$ as close as possible to a square. Instead of constructing the set $P_s$, we will construct the remaining graph $G'$ (as a subgraph of $G$). % and argue that $P_1 = V(G')-V(G_1)$.
    Let $t=n^2-k_1$ represent the number of vertices of $G'$. Note that $t=x^2+o$, where $o\le 2x$. We consider three cases.

    \textbf{Case 1:} $o = 0$. Then $t=x^2$ and we let $G'$ be the square grid on $x^2$ vertices, namely $V(G')=\{(i,j) \, : \, i,j \in [x]\}$. We infer that $m(G',2)=x = \sqrt{t} = \left\lceil\frac{1}{2}\lceil 2\sqrt{t} \, \rceil \right\rceil$. When $m = 8$, $n = 5$, the polluted $m \times n$ grid with $k=24$ is illustrated in Figure~\ref{fig:grid large k}(a) where here $x = 4$ and $t = 16 = x^2$ and where dotted vertices (associated with the set $P$ defined earlier) are polluted and shaded vertices present initial infection.

    \textbf{Case 2:} $0 < o  \le x$. Then let $V(G')= \{(i,j) \, : \, i,j \in [x]\} \cup \{(1,x+1),\ldots, (o,x+1)\}$. Namely we add $o$ vertices to a new row on top of the square grid from Case~1. Then the construction is similar to the one in the proof of Lemma~\ref{lem:rectangle grids}. Again let $C$ be the set of alternating black and white vertices from the first column and first row, starting with the black vertex $(1,x+1)$ and ending with the white vertex $(x,1)$ (since this set has length $2x$, it indeed ends in a white vertex). Let $S$ be the set of black vertices from $C$ together with $(x,1)$. Clearly, $|S|=x+1$ and it is easy to see that $S$ is a minimum percolating set of $G'$ and $m(G',2)=x+1$. Furthermore, since $0<o\le x$, it follows that $\left\lceil\frac{1}{2}\lceil 2\sqrt{t} \, \rceil\right\rceil = \left\lceil\frac{1}{2}(2x+1)\right\rceil = x+1$.  When $m = 8$, $n = 5$, the polluted $m \times n$ grid with $k=22$ is illustrated in Figure~\ref{fig:grid large k}(b) where here $x = 4$, $o = 2$ and $t = 18 = x^2 + o$ and where dotted vertices (associated with the set $P$ defined earlier) are polluted and shaded vertices present initial infection.

    \textbf{Case 3:} $x < o \le 2x.$ Then $$V(G')= \{(i,j) \, : \, i,j \in [x]\} \cup \{(1,x+1),\ldots, (x,x+1)\} \cup \{(x+1,1),\ldots,(x+1,o-x)\}.$$ Namely we add a complete row on top of the square grid from Case~1, while also adding vertices to the right side of the final column of the square grid. Similarly to Case~2, we let $C$ be the set of alternating black and white vertices from the first column and first row, starting with the black vertex $(1,x+1)$ and ending with the black vertex $(x+1,1)$ (now the set has length $2x+1$). Let $S$ be the set of black vertices from $C$, therefore $|S|=x+1$. This is again a minimum percolating set of $G'$ and $m(G',2)=x+1.$ Furthermore, since $x < o \le 2x$ it follows that $\left\lceil\frac{1}{2}\lceil 2\sqrt{t} \, \rceil\right\rceil = \left\lceil\frac{1}{2}(2x+2)\right\rceil=x+1$. This completes the proof of Lemma~\ref{lem:square grids}.
\end{proof}

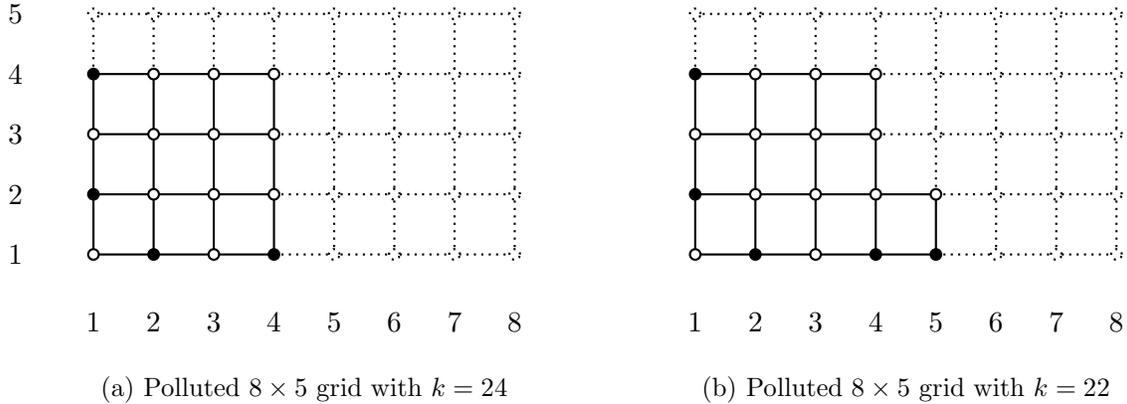
\begin{figure}[ht!]
\begin{center}
\begin{tikzpicture}[scale=0.8,style=thick,x=1cm,y=1cm]
\def\vr{2.5pt} % \vr = vertex radius;
% define vertices

\path (0,0.2) coordinate (1);
\path (1,0.2) coordinate (2);
\path (2,0.2) coordinate (3);
\path (3,0.2) coordinate (4);
\path (4,0.2) coordinate (5);
\path (5,0.2) coordinate (6);
\path (6,0.2) coordinate (7);
\path (7,0.2) coordinate (8);

\path (0,1) coordinate (a1);
\path (0,2) coordinate (b1);
\path (0,3) coordinate (c1);
\path (0,4) coordinate (d1);
\path (0,5) coordinate (e1);

\path (1,1) coordinate (a2);
\path (1,2) coordinate (b2);
\path (1,3) coordinate (c2);
\path (1,4) coordinate (d2);
\path (1,5) coordinate (e2);

\path (2,1) coordinate (a3);
\path (2,2) coordinate (b3);
\path (2,3) coordinate (c3);
\path (2,4) coordinate (d3);
\path (2,5) coordinate (e3);

\path (3,1) coordinate (a4);
\path (3,2) coordinate (b4);
\path (3,3) coordinate (c4);
\path (3,4) coordinate (d4);
\path (3,5) coordinate (e4);

\path (4,1) coordinate (a5);
\path (4,2) coordinate (b5);
\path (4,3) coordinate (c5);
\path (4,4) coordinate (d5);
\path (4,5) coordinate (e5);

\path (5,1) coordinate (a6);
\path (5,2) coordinate (b6);
\path (5,3) coordinate (c6);
\path (5,4) coordinate (d6);
\path (5,5) coordinate (e6);

\path (6,1) coordinate (a7);
\path (6,2) coordinate (b7);
\path (6,3) coordinate (c7);
\path (6,4) coordinate (d7);
\path (6,5) coordinate (e7);

\path (7,1) coordinate (a8);
\path (7,2) coordinate (b8);
\path (7,3) coordinate (c8);
\path (7,4) coordinate (d8);
\path (7,5) coordinate (e8);

%  edges
\draw (a1)--(b1)--(c1)--(d1);
\draw[dotted] (d1)--(e1);
\draw (a2)--(b2)--(c2)--(d2);
\draw[dotted] (d2)--(e2);
\draw (a3)--(b3)--(c3)--(d3);
\draw[dotted] (d3)--(e3);
\draw (a4)--(b4)--(c4)--(d4);
\draw[dotted] (d4)--(e4);
\draw[dotted] (a5)--(b5)--(c5)--(d5)--(e5);
\draw[dotted] (a6)--(b6)--(c6)--(d6)--(e6);
\draw[dotted] (a7)--(b7)--(c7)--(d7)--(e7);
\draw[dotted] (a8)--(b8)--(c8)--(d8)--(e8);

\draw (a1)--(a2)--(a3)--(a4);
\draw[dotted] (a4)--(a5)--(a6)--(a7)--(a8);
\draw (b1)--(b2)--(b3)--(b4);
\draw[dotted] (b4)--(b5)--(b6)--(b7)--(b8);
\draw (c1)--(c2)--(c3)--(c4);
\draw[dotted] (c4)--(c5)--(c6)--(c7)--(c8);
\draw (d1)--(d2)--(d3)--(d4);
\draw[dotted] (d4)--(d5)--(d6)--(d7)--(d8);
\draw[dotted] (e1)--(e2)--(e3)--(e4)--(e5)--(e6)--(e7)--(e8);

% vertices
\draw (a1) [fill=white] circle (\vr);
\draw (b1) [fill=black] circle (\vr);
\draw (c1) [fill=white] circle (\vr);
\draw (d1) [fill=black] circle (\vr);
\draw[dotted] (e1) [fill=white] circle (\vr);

\draw (a2) [fill=black] circle (\vr);
\draw (b2) [fill=white] circle (\vr);
\draw (c2) [fill=white] circle (\vr);
\draw (d2) [fill=white] circle (\vr);
\draw[dotted] (e2) [fill=white] circle (\vr);

\draw (a3) [fill=white] circle (\vr);
\draw (b3) [fill=white] circle (\vr);
\draw (c3) [fill=white] circle (\vr);
\draw (d3) [fill=white] circle (\vr);
\draw[dotted] (e3) [fill=white] circle (\vr);

\draw (a4) [fill=black] circle (\vr);
\draw (b4) [fill=white] circle (\vr);
\draw (c4) [fill=white] circle (\vr);
\draw (d4) [fill=white] circle (\vr);
\draw[dotted] (e4) [fill=white] circle (\vr);

\draw[dotted] (a5) [fill=white] circle (\vr);
\draw[dotted] (b5) [fill=white] circle (\vr);
\draw[dotted] (c5) [fill=white] circle (\vr);
\draw[dotted] (d5) [fill=white] circle (\vr);
\draw[dotted] (e5) [fill=white] circle (\vr);

\draw[dotted] (a6) [fill=white] circle (\vr);
\draw[dotted] (b6) [fill=white] circle (\vr);
\draw[dotted] (c6) [fill=white] circle (\vr);
\draw[dotted] (d6) [fill=white] circle (\vr);
\draw[dotted] (e6) [fill=white] circle (\vr);

\draw[dotted] (a7) [fill=white] circle (\vr);
\draw[dotted] (b7) [fill=white] circle (\vr);
\draw[dotted] (c7) [fill=white] circle (\vr);
\draw[dotted] (d7) [fill=white] circle (\vr);
\draw[dotted] (e7) [fill=white] circle (\vr);

\draw[dotted] (a8) [fill=white] circle (\vr);
\draw[dotted] (b8) [fill=white] circle (\vr);
\draw[dotted] (c8) [fill=white] circle (\vr);
\draw[dotted] (d8) [fill=white] circle (\vr);
\draw[dotted] (e8) [fill=white] circle (\vr);

% text

\draw[anchor = north] (1) node {1};
\draw[anchor = north] (2) node {2};
\draw[anchor = north] (3) node {3};
\draw[anchor = north] (4) node {4};
\draw[anchor = north] (5) node {5};
\draw[anchor = north] (6) node {6};
\draw[anchor = north] (7) node {7};
\draw[anchor = north] (8) node {8};

\draw[anchor = east] (a) node {1};
\draw[anchor = east] (b) node {2};
\draw[anchor = east] (c) node {3};
\draw[anchor = east] (d) node {4};
\draw[anchor = east] (e) node {5};

%%%%%%%%%
\draw (3.5,-1.25) node {{\small (a) Polluted $8\times 5$ grid with $k=24$}};
%%%%%%%%

%%%%%%%%%%%%%%%%%%%%%%%%%% SECOND GRID%%%%%%%%%%%%%%%%%%%%%%%%%%%%%%%

\begin{scope}[shift={(10,0)}]

    \path (0,0.2) coordinate (1);
    \path (1,0.2) coordinate (2);
    \path (2,0.2) coordinate (3);
    \path (3,0.2) coordinate (4);
    \path (4,0.2) coordinate (5);
    \path (5,0.2) coordinate (6);
    \path (6,0.2) coordinate (7);
    \path (7,0.2) coordinate (8);

    \path (0,1) coordinate (a1);
    \path (0,2) coordinate (b1);
    \path (0,3) coordinate (c1);
    \path (0,4) coordinate (d1);
    \path (0,5) coordinate (e1);

    \path (1,1) coordinate (a2);
    \path (1,2) coordinate (b2);
    \path (1,3) coordinate (c2);
    \path (1,4) coordinate (d2);
    \path (1,5) coordinate (e2);

    \path (2,1) coordinate (a3);
    \path (2,2) coordinate (b3);
    \path (2,3) coordinate (c3);
    \path (2,4) coordinate (d3);
    \path (2,5) coordinate (e3);

    \path (3,1) coordinate (a4);
    \path (3,2) coordinate (b4);
    \path (3,3) coordinate (c4);
    \path (3,4) coordinate (d4);
    \path (3,5) coordinate (e4);

    \path (4,1) coordinate (a5);
    \path (4,2) coordinate (b5);
    \path (4,3) coordinate (c5);
    \path (4,4) coordinate (d5);
    \path (4,5) coordinate (e5);

    \path (5,1) coordinate (a6);
    \path (5,2) coordinate (b6);
    \path (5,3) coordinate (c6);
    \path (5,4) coordinate (d6);
    \path (5,5) coordinate (e6);

    \path (6,1) coordinate (a7);
    \path (6,2) coordinate (b7);
    \path (6,3) coordinate (c7);
    \path (6,4) coordinate (d7);
    \path (6,5) coordinate (e7);

    \path (7,1) coordinate (a8);
    \path (7,2) coordinate (b8);
    \path (7,3) coordinate (c8);
    \path (7,4) coordinate (d8);
    \path (7,5) coordinate (e8);

    %  edges
    \draw (a1)--(b1)--(c1)--(d1);
    \draw[dotted] (d1)--(e1);
    \draw (a2)--(b2)--(c2)--(d2);
    \draw[dotted] (d2)--(e2);
    \draw (a3)--(b3)--(c3)--(d3);
    \draw[dotted] (d3)--(e3);
    \draw (a4)--(b4)--(c4)--(d4);
    \draw[dotted] (d4)--(e4);
    \draw (a5)--(b5);
    \draw[dotted] (b5)--(c5)--(d5)--(e5);
    \draw[dotted] (a6)--(b6)--(c6)--(d6)--(e6);
    \draw[dotted] (a7)--(b7)--(c7)--(d7)--(e7);
    \draw[dotted] (a8)--(b8)--(c8)--(d8)--(e8);

    \draw (a1)--(a2)--(a3)--(a4)--(a5);
    \draw[dotted] (a5)--(a6)--(a7)--(a8);
    \draw (b1)--(b2)--(b3)--(b4)--(b5);
    \draw[dotted] (b5)--(b6)--(b7)--(b8);
    \draw (c1)--(c2)--(c3)--(c4);
    \draw[dotted] (c4)--(c5)--(c6)--(c7)--(c8);
    \draw (d1)--(d2)--(d3)--(d4);
    \draw[dotted] (d4)--(d5)--(d6)--(d7)--(d8);
    \draw[dotted] (e1)--(e2)--(e3)--(e4)--(e5)--(e6)--(e7)--(e8);

    % vertices
    \draw (a1) [fill=white] circle (\vr);
    \draw (b1) [fill=black] circle (\vr);
    \draw (c1) [fill=white] circle (\vr);
    \draw (d1) [fill=black] circle (\vr);
    \draw[dotted] (e1) [fill=white] circle (\vr);

    \draw (a2) [fill=black] circle (\vr);
    \draw (b2) [fill=white] circle (\vr);
    \draw (c2) [fill=white] circle (\vr);
    \draw (d2) [fill=white] circle (\vr);
    \draw[dotted] (e2) [fill=white] circle (\vr);

    \draw (a3) [fill=white] circle (\vr);
    \draw (b3) [fill=white] circle (\vr);
    \draw (c3) [fill=white] circle (\vr);
    \draw (d3) [fill=white] circle (\vr);
    \draw[dotted] (e3) [fill=white] circle (\vr);

    \draw (a4) [fill=black] circle (\vr);
    \draw (b4) [fill=white] circle (\vr);
    \draw (c4) [fill=white] circle (\vr);
    \draw (d4) [fill=white] circle (\vr);
    \draw[dotted] (e4) [fill=white] circle (\vr);

    \draw (a5) [fill=black] circle (\vr);
    \draw (b5) [fill=white] circle (\vr);
    \draw[dotted] (c5) [fill=white] circle (\vr);
    \draw[dotted] (d5) [fill=white] circle (\vr);
    \draw[dotted] (e5) [fill=white] circle (\vr);

    \draw[dotted] (a6) [fill=white] circle (\vr);
    \draw[dotted] (b6) [fill=white] circle (\vr);
    \draw[dotted] (c6) [fill=white] circle (\vr);
    \draw[dotted] (d6) [fill=white] circle (\vr);
    \draw[dotted] (e6) [fill=white] circle (\vr);

    \draw[dotted] (a7) [fill=white] circle (\vr);
    \draw[dotted] (b7) [fill=white] circle (\vr);
    \draw[dotted] (c7) [fill=white] circle (\vr);
    \draw[dotted] (d7) [fill=white] circle (\vr);
    \draw[dotted] (e7) [fill=white] circle (\vr);

    \draw[dotted] (a8) [fill=white] circle (\vr);
    \draw[dotted] (b8) [fill=white] circle (\vr);
    \draw[dotted] (c8) [fill=white] circle (\vr);
    \draw[dotted] (d8) [fill=white] circle (\vr);
    \draw[dotted] (e8) [fill=white] circle (\vr);

    % text

    \draw[anchor = north] (1) node {1};
    \draw[anchor = north] (2) node {2};
    \draw[anchor = north] (3) node {3};
    \draw[anchor = north] (4) node {4};
    \draw[anchor = north] (5) node {5};
    \draw[anchor = north] (6) node {6};
    \draw[anchor = north] (7) node {7};
    \draw[anchor = north] (8) node {8};

%%%%%%%%%
\draw (3.5,-1.25) node {{\small (b) Polluted $8\times 5$ grid with $k=22$}};
%%%%%%%%

\end{scope}

\end{tikzpicture}
\end{center}

\caption{Polluted $8\times 5$ grid with $k=24$ and $k=22$, respectively, in the proof of Lemma~\ref{lem:square grids}}
\label{fig:grid large k}
\end{figure}

From the proof of Lemma~\ref{lem:square grids}, we also infer the equality between  the upper bound in Lemma~\ref{lem:square grids}, and the lower bounds in the second and third line of Lemma~\ref{lem:rectangle}. Notably
$$\left\lceil\frac{1}{2}\left\lceil 2\sqrt{t}\right\rceil\right\rceil=x
\textrm{ if } t=x^2,$$ and
$$\left\lceil\frac{1}{2}\left\lceil 2\sqrt{t}\right\rceil\right\rceil=x+1
\textrm{ if } t\ne x^2.$$
Combining this fact and earlier auxiliary results, we infer the main result of this section, and let us recall its formulation.

\noindent {\bf{Theorem 1.}}
    If $G=P_m \square P_n$ is the grid of size $m \times n$ where $2\le n \le m$, then
    \[
        m_k^{\min}(G,2) =
        \begin{cases}
          \left\lceil\frac{n+m-\left\lfloor\frac{k}{n}\right\rfloor}{2}\right\rceil, & \text{if } 1\le k \le (m-n)n; \2 \\
          \left\lceil\frac{\left\lceil 2\sqrt{mn-k}\right\rceil}{2}\right\rceil, & \text{if } (m-n)n\le k \le mn.
        \end{cases}
    \]

Note that the bound in the first row of Theorem~\ref{thm:main-grid} can be written in a different form by expressing the value that needs to be subtracted from $m(G,2)$.

\begin{remark} Let $n,m \ge 2$ and $G=P_m \square P_n$ be the Cartesian grid of size $m \times n$ where $n \le m$. If $k\le (m-n)n $ and $\ell = \lfloor\frac{k}{n}\rfloor$, then
     \[ m_k^{\min}(G,2) = \left\lceil\frac{m+n-\ell}{2}\right\rceil =
         \begin{cases}
          m(G,2) - \left\lfloor\frac{1}{2}\lfloor\frac{k}{n}\rfloor \right\rfloor, & \text{if } n+m \text{ even}; \2 \\
          m(G,2)-\left\lfloor\frac{1}{2}(\lfloor\frac{k}{n}\rfloor+1)\right\rfloor, & \text{if } n+m \text{ odd.}
        \end{cases}
    \]
\end{remark}
\begin{proof}
Consider the two cases with respect to the parity of $m+n$. If $n+m$ is even, then 
\[
m(G,2) - \left\lfloor \frac{1}{2} \left\lfloor\frac{k}{n} \right\rfloor \right\rfloor = \frac{m+n}{2} - \left\lfloor \frac{\ell}{2} \right\rfloor =  \left\lceil \frac{m+n-\ell}{2} \right\rceil.
\]
If $n+m$ is odd, then 
\[
m(G,2)-\left\lfloor \frac{1}{2} \left( \left\lfloor\frac{k}{n} \right\rfloor+1 \right) \right\rfloor = \frac{m+n+1}{2} - \left\lfloor \frac{\ell+1}{2} \right\rfloor =  \left\lceil\frac{m+n+1-(\ell+1)}{2}\right\rceil =\left\lceil\frac{m+n-\ell}{2}\right\rceil.
\]
\end{proof}

In the last result of this section, we consider the maximum possible value of $m(G-A,2)$, where $G$ is a square grid and $A$ is a subset of $V(G)$ with $|A|=k$. We restrict our discussion on the sets $A$ with $|A|\le \lceil\frac{(n-2)(m-2)}{2}\rceil$ the reason being that in this way $A$ can be an independent set in $G$ whose vertices all have degree $4$ in $G$. Note that by Theorem~\ref{thm:general} the $r$-neighbor bootstrap percolation does not decrease when an independent set is removed, and at the same time removing an independent set yields a natural lower bound on the parameter $m_k^{\max}(G,2)$ in grids $G$.

\begin{theorem}
   If $G=P_m \square P_n$ is an $m \times n$ grid and $k \le \left\lceil\frac{(n-2)(m-2)}{2}\right\rceil$, then 
   \[
   m_k^{\max}(G,2) \ge m(G,2)+k.
   \]
\end{theorem}
\begin{proof}
    Using the perimeter idea again, we deduce that whenever a vertex of degree $4$ is removed from $G$, so that the total perimeter is increased by $4$, we require at least one additional vertex in the initial infection. This is achievable by removing an independent set of vertices of degree~$4$ (removing two adjacent vertices increases the perimeter by~$6$). Since there are $(m-2)(n-2)$ vertices of degree~$4$ it is easy to see that the size of the largest independent set of such vertices is $\left\lceil\frac{(n-2)(m-2)}{2}\right\rceil$, completing the proof.
\end{proof}

\section{Concluding remarks}

In Theorem~\ref{thm:main-grid} we obtained the closed formula for $m_k^{\min}(G,2)$, where $G$ is an arbitrary Cartesian grid. For the parameter $m_k^{\max}(P_m \cp P_n,2)$ we obtained a natural lower bound expressed in terms of $m(P_m\cp P_n,2)$. We suspect the lower bound could be the exact value and we pose it as an open problem.

Note that the torus $C_m \square C_n$ is a vertex-transitive $4$-regular graph. Therefore, $m(G-x,2)=m(G-y,2)$ for any two vertices $x,y$. This implies $m_1^{\min}(G,2) = m_1^{\max}(G,2)$. We immediately infer that if $G=C_m \cp C_n$ and $x \in V(G)$, then 
\[
m(G-x,2)=m(G,2)=\left\lceil\frac{m+n}{2}\right\rceil-1.
\]
It would be interesting to determine the effect of pollution on the bootstrap percolation numbers on the torus $C_m \cp C_n$. We envision that new techniques might be needed to determine the values of $m_k^{\min}(C_m \cp C_n,2)$ and $m_k^{\max}(C_m \cp C_n,2)$.

\section*{Acknowledgements}

Bo\v{s}tjan Bre\v{s}ar was supported by the Slovenian Research and Innovation Agency (ARIS) under the grants P1-0297, N1-0285, and J1-4008. Jaka Hed\v zet was supported by the Slovenian Research and Innovation Agency (ARIS) under the grant P1-0297. Michael A. Henning was supported in part by the University of Johannesburg.

%%%%

\end{document}